\documentclass[a4paper,10pt]{article}
\usepackage[utf8]{inputenc}
\usepackage{amsmath,amsfonts,amssymb,amsthm,enumerate}
\usepackage[mathscr]{euscript}

\usepackage{authblk}
\usepackage[english]{babel}
\usepackage[dvips]{graphicx}
\usepackage[pdftex]{hyperref}

\newtheorem{theorem}{Theorem}[section]

\newtheorem{definition}{Definition}[section]
\newtheorem{example}{Example}[section]

\newtheorem{lemma}{Lemma}[section]

\newtheorem{proposition}{Proposition}[section]
\newtheorem{remark}{Remark}[section]

\newcommand{\diam}{\operatorname{diam}}

\newcommand{\V}{\mathcal{V}}

\newcommand{\U}{\mathcal{U}}

\newcommand{\bc}{\begin{center}}
	\newcommand{\ec}{\end{center}}

\def\p{{\mathbb{P}}}
\def\N{{\mathbb{N}}}

\def\Z{{\mathbb{Z}}}

\def\||{\parallel}
\def\ds{\displaystyle}
\def\dfrac{\ds \frac}

\title{Topological Stability of Random $cw$-Expansive Systems with Shadowing}

\author[1]{M. Oliveira}
\author[2]{R. Bilbao}
\author[3]{E. Santana}

\affil[1]{Universidade Estadual do Maranhão, São Luís, Brazil}
\affil[2]{Universidad Pedagógica y Tecnológica de Colombia, Boyacá, Colombia}
\affil[3]{Universidade Federal de Alagoas, Penedo, Brazil}

\begin{document}
\maketitle

\begin{abstract}
This paper investigates the topological stability of random dynamical systems . Our main goal is to extend the classical result of Walters \cite{Walters77} to the random setting by employing the notions of random continuum-wise expansivity and shadowing property. We prove that any random dynamical system that is random $cw$-expansive and satisfies the random shadowing property is randomly topologically stable. Furthermore, in the random setting we establish the topological invariance of these properties under conjugacy and analyze the relationship between topological transitivity and the periodic shadowing property.
\end{abstract}

\section{Introduction}\label{Introduction}
\vspace{0.50cm}

Expansiveness and shadowing are fundamental dynamical properties in the study of systems defined on metric spaces. These notions play a central role in several areas of mathematics, including topological dynamics, ergodic theory, symbolic dynamics, and their interactions with probability theory. A detailed account of expansiveness and shadowing in deterministic settings can be found in \cite{Walters78, AokiHiraide, pilyugin}.

The study of stability in dynamical systems has been a central topic since the foundational works of Anosov and Bowen. Topological stability, in particular, describes the robustness of the qualitative behavior of a system under small perturbations. An important result in this direction was established by Walters \cite{Walters77}, who proved that a homeomorphism on a compact metric space possessing both expansiveness and the shadowing property is topologically stable. Since then, topological stability has been extensively investigated in different contexts, revealing deep connections with hyperbolicity, Axiom A systems, and structural stability \cite{Robbin, Robinson}.

The concept of expansiveness, originally introduced for homeomorphisms on compact metric spaces \cite{Utz}, has several important generalizations. Among these are pointwise expansiveness \cite{Reddy}, entropy-expansiveness \cite{Bowen}, measure expansiveness \cite{MoralesPacifico}, and $n$-expansiveness \cite{Morales}. A particularly relevant extension is the notion of continuum-wise expansiveness ($cw$-expansiveness), introduced by Kato \cite{Kato}. In contrast classical expansiveness, which concerns the separation of individual points, $cw$-expansiveness only requires that the diameter of every non-trivial continuum eventually exceeds a fixed positive constant. This notion provides a meaningful weakening of classical expansiveness while preserving several important dynamical features.

In recent decades, considerable attempt has been devoted to extending deterministic dynamical properties to the framework of Random Dynamical Systems (RDS). As systematically developed by Arnold \cite{Arnold}, RDS provide a robust mathematical framework for describing systems influenced by external noise or random perturbations, naturally arising in many applications. Within this setting, classical notions such as expansiveness, hyperbolicity, and shadowing have been adapted to the random context; see, for instance, \cite{Kifer2000, GundlachKifer, Liu98}.

Random expansiveness was introduced to describe systems in which nearby orbits separate along random trajectories \cite{Kifer2000}. More recently, the notion of random $cw$-expansiveness was proposed as a natural extension of continuum-wise expansiveness to the random setting \cite{BilbaoOliveiraSantana2024}, allowing expansive behavior to be studied along continua under random dynamics. However, despite its recent introduction, the dynamical consequences of random $cw$-expansiveness remain largely unexplored, particularly with respect to robustness and stability under perturbations.

Recently, important advances in the stability theory of random dynamical systems were obtained in \cite{Azjargal}, where it was shown that the combination of $\omega$-expansiveness and the shadowing property implies random topological stability. Motivated by Walters' classical theorem \cite{Walters77}, the main purpose of this work is to develop a stability theory for random $cw$-expansive systems. Our work extends this result by replacing the assumption of random $\omega$-expansiveness with the more general notion of random $cw$-expansiveness. Since $cw$-expansiveness is strictly weaker than pointwise expansiveness, our approach applies to a broader class of random dynamical systems while preserving topological stability under the shadowing property.

More precisely, we prove that random dynamical systems that are random $cw$-expansive and satisfy the random shadowing property are randomly topologically stable, there by establishing a random counterpart of Walters’ theorem in the continuum-wise setting. Since continuum-wise expansiveness is strictly weaker than classical expansiveness, our results apply to a broader class of random dynamical systems for which classical expansiveness may fail to hold. In this way, our work substantially enlarges the scope of random topological stability results.

Beyond stability, we investigate structural and topological aspects of this class of systems. In particular, we prove that random $cw$-expansiveness and the random shadowing property are preserved under random conjugacies, establish uniqueness of shadows, and analyze the relationship between topological transitivity and the periodic shadowing property in the random setting. The latter is especially relevant for understanding the approximation of recurrent behavior and periodic structures in random dynamics.

The paper is organized as follows. In Section 2, we present the preliminaries on random dynamical systems, bundle maps, and random conjugacies. Section 3 introduces random cw-expansiveness and establishes its main properties, including invariance under conjugacy and its relation with sensitivity. In Section 4, we study random shadowing, uniqueness of shadows, periodic shadowing, and prove the random topological stability theorem for cw-expansive systems with shadowing. Finally, Section 5 provides examples illustrating the main results.

\section{Random Dynamical Systems}

Let $(X,d)$ be a compact metric space and let $\mathcal{B}$ denote its
Borel $\sigma$-algebra. Let $(\Omega,\mathcal{F},\mathbb{P})$ be a
probability space and let
$\theta:\Omega\to\Omega$ be an invertible, measurable and
$\mathbb P$-preserving transformation.

Let $\mathcal U\subset\Omega\times X$ be an $\mathcal F\otimes\mathcal B$-measurable set whose fibers 
\[
U_w=\{x\in X:(w,x)\in\mathcal U\},
\]
are nonempty compact subsets of $X$ for every $w\in\Omega$.

A bundle random homeomorphism
$f=\{f_w\}_{w\in\Omega}$ over
$(\Omega,\mathcal F,\mathbb P,\theta)$
is given by a family of homeomorphisms
$f_w:U_w\longrightarrow U_{\theta(w)}$, 
such that the maps $(w,x)\longmapsto f_w(x)$ and $(w,x)\longmapsto f_w^{-1}(x)$ are jointly measurable. 

For each $w\in\Omega$, define $f_w^0=\mathrm{Id}$, and for every $n\ge1$, \[f_w^n=
f_{\theta^{n-1}(w)}
\circ
f_{\theta^{n-2}(w)}
\circ\cdots\circ
f_w.\]

For $n\ge1$, define
\[
f_w^{-n}
=
(f_{\theta^{-n}(w)})^{-1}
\circ
(f_{\theta^{-n+1}(w)})^{-1}
\circ\cdots\circ
(f_{\theta^{-1}(w)})^{-1}.
\]

The associated skew-product map
\[
F:\mathcal U\longrightarrow\mathcal U,
\]
is defined by
\[
F(w,x)=
(\theta(w),f_w(x)).
\]

\begin{definition}
A random set in $\mathcal U$ is a measurable subset
$A\subset\mathcal U$. For each $w\in\Omega$, its fiber is defined by
\[
A_w=\{x\in U_w:(w,x)\in A\}.
\]

We say that $A$ is a random open set if $A_w$ is an open subset of
$U_w$ for every $w\in\Omega$.
\end{definition}

Throughout the paper, the abbreviation RDS stands for random dynamical system.

Topological conjugacy is a fundamental notion for the qualitative equivalence of dynamical systems. In the random setting, this concept incorporates the dependence on the underlying driving system $\theta$. Random topological conjugacy was developed by Liu \cite{Liu98} in the context of random perturbations of hyperbolic systems. We formally define this equivalence.

\begin{definition}
Let
\(F:\mathcal U\to\mathcal U\)
and
\(G:\mathcal V\to\mathcal V\)
be random dynamical systems over
\((\Omega,\mathcal F,\mathbb P,\theta)\).

We say that \(F_{\mathcal U}\) and \(G_{\mathcal V}\) are semiconjugate if there exists a measurable bundle map $\pi(w,x) \to (w,\pi_w(x))$, whose fiber maps $\pi_w:U_w\to V_w$ are continuous surjections and satisfy $\pi_{\theta(w)}\circ f_w
=
g_w\circ\pi_w$ for every \(w\in\Omega\).

If each \(\pi_w\) is a homeomorphism, then \(F_{\mathcal U}\) and \(G_{\mathcal V}\) are said to be conjugate.
\end{definition}

A concrete example illustrating this construction can be found in Liu~\cite{Liu98}, Example 2.5.

\section{Random cw-expansivity}
In this section, we introduce random continuum-wise expansiveness, a random counterpart of the notion
introduced by Kato \cite{Kato}, and establish some of its basic
properties.

For any positive random variable $\delta=\delta(w)$,  points $x,y\in U_{w}$ and numbers $m,n\in \N$ with $m<n$, we introduce a family of metrics on each fiber $U_{w}$ by 
\[
d^{w}_{\delta; \, m,n}(x,y) := \max_{m\leq k \leq n}\left\lbrace \dfrac{d(f^{k}_{w}(x),f^{k}_{w}(y))}{\delta(\theta^{k}(w))} \right\rbrace.
\]
We denote $d^{w}_{\delta; n} = d^{w}_{\delta; \, 0,n}$ and $d^{w}_{\delta; \pm n} = d^{w}_{\delta; \, -n,n}$. Given a point $x\in U_{w}$ we denote by $B_{w}[x,\delta,n]$ and $B_{w}[x,\delta,\pm n]$ the closed unit balls with respect to the metrics $d^{w}_{\delta; \, n}$ and $d^{w}_{\delta; \, \pm n}$, that is, 
\[
B_{w}[x,\delta,n] = \{y\in U_{w}; \ d(f^{k}_{w}(x),f^{k}_{w}(y))\leq \delta(\theta^{k}(w)), \  0\leq k \leq n-1  \},
\]
and
\[
B_{w}[x,\delta,\pm n] = \{y\in U_{w}; \ d(f^{k}_{w}(x),f^{k}_{w}(y))\leq \delta(\theta^{k}(w)), \ -n + 1 \leq k \leq n-1  \}.
\]
We write $B_{w}(x,\delta) := B_{w}[x,\delta,1]$ and denote $\Gamma_{\delta}(w,x) = \bigcap_{n\geq 1} B_{w}[x,\delta,\pm n] $ which consists of points whose entire orbits remain
$\delta$-close to the orbit of $x$.

The notion of random expansiveness was introduced in \cite{Kifer2000}.
\begin{definition}
A random dynamical system $F_{\mathcal{U}}$ is said to be random expansive  if there exists a positive random variable $\delta=\delta(w)$, called an expansivity characteristic, such that $\Gamma_{\delta}(w,x) = \{x\}$ for $\mathbb P$-almost every $w\in\Omega$ and every $x\in U_w$. 
\end{definition}

The notion of continuum-wise expansiveness was extended to the random setting in \cite{BilbaoOliveiraSantana2024}. Recall that a \emph{continuum} is a nonempty compact connected set. A continuum is said to be \emph{nontrivial} if it contains more than one point.

\begin{definition}
A random dynamical system  $F_{\U}$ will be called random continuum-wise expansive if there exists a  positive mensurable random variable $\delta=\delta(w)$ such that for $\mathbb{P}$-a.e $w \in \Omega$ and every non-trivial continuum $C \subset U_{w}$ there exists $n \in \mathbb{Z}$ such that $\text{diam}(f_{w}^{n}(C)) > \delta(\theta^{n}(w))$. The random variable $\delta$ is called a $cw$-expansivity characteristic.  
\end{definition}

The following lemma provides a useful characterization of random $cw$-expansivity. 
\begin{lemma}[\cite{BilbaoOliveiraSantana2024}]\label{lemma1}
The following statements are equivalent:
\begin{itemize}
  \item[(i)] The random dynamical system $F_{\U}$ is random continuum-wise expansive. 
  
  \item[(ii)] There exists a random variable $\delta : \Omega \to (0,\infty)$ such that for $\mathbb{P}$-a.e $w\in \Omega$ and all $x\in U_{w}$, the set $\Gamma_{\delta}(w,x)$ contains no nontrivial continuum.
\end{itemize}
\end{lemma}

The following result establishes that random $cw$-expansiveness is a topological invariant, being preserved under random conjugacy.

\begin{proposition}
Let $F_{\U}$ and $G_{\V}$ be two conjugate bundle random dynamical systems. If $F_{\U}$ is random cw-expansive, then so is $G_{\V}$.
\end{proposition}
\begin{proof}
Let $\varepsilon:\Omega\to(0,\infty)$ be a cw-expansivity characteristic for $F_{\mathcal U}$. By hypothesis, there exists a family of homeomorphisms $\{\pi_w: U_w \to V_w\}_{w\in \Omega}$ such that $f_{w}^n\circ \pi_{w}^{-1} = \pi^{-1}_{\theta^{n}(w)}\circ g^{n}_{w}$. 
Since each fiber $V_w$ is compact and $\pi_w^{-1}:V_w\to U_w$ is uniformly continuous, for every $\varepsilon(w)>0$ there exists $\alpha(w)>0$ such that if $d(x,y)<\alpha(w)$, then $d(\pi_w^{-1}(x),\pi_w^{-1}(y))
<
\varepsilon(w)$ for all $x,y\in V_w$. Moreover, since $w\mapsto\pi_w$ is a measurable bundle map and the fibers are compact, we may choose $\alpha:\Omega\to(0,\infty)$ to be measurable.

Fix $x\in V_w$ and let
$\Delta_w\subset\Gamma_\alpha^G(w,x)$ be a continuum. Given $y\in \Delta_w$, we have by conjugacy and uniform continuity
\begin{eqnarray}
d(f^{n}_{w}(\pi_{w}^{-1}(x)),f^{n}_{w}(\pi_{w}^{-1}(y))) & = &d(\pi_{\theta^{n}(w)}^{-1}(g^{n}_{w}(x)),\pi_{\theta^{n}(w)}^{-1}(g^{n}_{w}(y))  \nonumber \\
 &\leq & \epsilon (\theta^{n}(w)).
\end{eqnarray}
This shows that the set $\pi_w^{-1}(\Delta_w)$ is a continuum contained in $\Gamma_{\epsilon}^F(w, \pi_w^{-1}(x))$. Since $F_{\U}$ is random cw-expansive, Lemma \ref{lemma1} implies that $\pi_w^{-1}(\Delta_w)$ must be a singleton. Because $\pi_w$ is a homeomorphism, $\Delta_w$ must also be a singleton. Therefore, $G_{\V}$ is random cw-expansive with characteristic $\alpha(w)$.
\end{proof}

The next lemma shows that every continuum with diameter bounded away from zero must exceed the expansivity scale after finitely many iterates.

\begin{lemma}\label{lemma2}
Let $F_{\mathcal U}$ be a random cw-expansive system with expansivity characteristic $e = e(w)$. Given a positive random variable $\varepsilon = \varepsilon(w)$, for $\mathbb P$-almost every $w\in\Omega$ there exists an integer $n(w)\ge1$ such that
\[
\max_{|i|\le n(w)}
\left\{
\operatorname{diam}(f_w^i(A))
-
\frac12e(\theta^i(w))
\right\}
>0
\]
for every continuum $A\subset U_w$
with $\operatorname{diam}(A)\ge\varepsilon(w)$.
\end{lemma}

\begin{proof}
Fix a positive random variable $\varepsilon=\varepsilon(w)$ and let $w$ belong to a full measure subset of $\Omega$. Assume, by contradiction, that the conclusion fails for this $w$. Then, for every integer $n\ge1$, there exists a continuum $A_{n,w}\subset U_w$ such that \[ \operatorname{diam}(A_{n,w})\ge\varepsilon(w) \] and \[ \operatorname{diam}(f_w^i(A_{n,w})) \le \frac12\,e(\theta^i(w)), \qquad |i|\le n. \]

Since each fiber $U_w$ is compact, the hyperspace $\mathcal{K}(U_w)$ of nonempty compact subsets of $U_w$, endowed with the Hausdorff metric, is also compact. By passing to a subsequence if necessary, we can assume that $A_{n,w}$ converges in the Hausdorff topology to a compact set $A_w \subset U_w$ in the Hausdorff topology as $n \to \infty$. Furthermore, since the family of continua is closed under the Hausdorff metric, $A_w$ is a continuum.

By the continuity of the diameter map with respect to the Hausdorff metric, we obtain
\[
\diam(A_w) = \lim_{n \to \infty} \diam(A_{n,w}) \ge \varepsilon(w) > 0,
\]
which implies that $A_w$ is a non-trivial continuum.

Now, fix an arbitrary integer $i \in \mathbb{Z}$. For all sufficiently large $n$ such that $n \ge |i|$, the assumption implies that $\diam(f_w^i(A_{n,w})) \le \frac{e(\theta^i(w))}{2}$. Using the continuity of the map $f_w^i$ and the diameter operator once more, it follows that
\[
\diam\left(f_w^i(A_w)\right) = \lim_{n \to \infty} \diam\left(f_w^i(A_{n,w})\right) \le \frac{e(\theta^i(w))}{2} < e(\theta^i(w)).
\]

Consequently, for any fixed base point $x \in A_w$ and for every $z \in A_w$, we have
\[
d\left(f_w^i(z), f_w^i(x)\right) \le \diam\left(f_w^i(A_w)\right) \le \frac{e(\theta^i(w))}{2} \le e(\theta^i(w)),
\]
for all $i \in \mathbb{Z}$. This yields
\[
A_w \subset \Gamma_e(w,x).
\]

Since $A_w$ is a non-trivial continuum contained in $\Gamma_e(w,x)$, this directly contradicts Lemma~\ref{lemma1}. Therefore, for $\mathbb P$-almost every $w\in\Omega$ there exists an integer $n(w)\ge1$ satisfying the conclusion.
\end{proof}

Motivated by Obaid and Kadhim~\cite{obaid}, we introduce a random version of equicontinuity for random dynamical systems.

\begin{definition}
A random dynamical system $F_{\mathcal U}$ is said to be random equicontinuousif for every positive random variable $\epsilon = \epsilon(w)$, there exists a positive measurable random variable $\delta = \delta(w)$ such that for $\mathbb{P}$-almost every $w\in \Omega$, for every $x,y\in U_w$.
w satisfying $d(x, y) < \delta(w)$, the following holds for all $n \in \mathbb{Z}$:
\[
d\left(f^{n}_{w}(x), f_{w}^{n}(y)\right) < \epsilon(\theta^n(w)),
\]
for $\mathbb{P}$-a.e. $w \in \Omega$.
\end{definition}

We now show that random $cw$-expansiveness and random equicontinuity are mutually exclusive properties on nontrivial spaces.

\begin{proposition}
Let $F_{\mathcal U}$ be a random dynamical system such that, for
$\mathbb P$-almost every $w\in\Omega$, the fiber $U_w$ is an
uncountable locally connected  Lindelöf metric space.
If $F_{\mathcal U}$ is random equicontinuous, then $F_{\mathcal U}$ is not random continuum-wise expansive.
\end{proposition}

\begin{proof}
Suppose, by contradiction, that $F_{\mathcal{U}}$ is random $cw$-expansive with expansivity characteristic $e = e(w)$. 

By applying the definition of random equicontinuity to the positive random variable $e(w) > 0$, there exists a positive random variable $\delta = \delta(w)$ such that, for $\mathbb{P}$-a.e. $w \in \Omega$, 
\[
d\left(f^{n}_{w}(x), f_{w}^{n}(y)\right) < e(\theta^n w),
\]
for all $n \in \mathbb{Z}$, whenever $x, y \in U_w$ satisfy $d(x,y) < \delta(w)$. This immediately implies that for every $x \in U_w$ we have
\[
B_w(x, \delta(w)) \subset B_w[x, \delta(w)] \subset \Gamma_{e}(w,x).
\]

By Lemma~\ref{lemma1}, the set $\Gamma_{e}(w,x)$ cannot contain any non-trivial continuum. Consequently, the ball $B_w[x, \delta(w)]$ is a singleton.  Since $U_w$ is a Lindelöf space for $\mathbb{P}$-a.e. $w \in \Omega$, the open cover $\bigcup_{x\in U_w} B_w(x, \delta(w))$ admits a countable subcover, which implies that
\[
U_w = \bigcup_{i \in \mathbb{N}} B_w(x_i, \delta(w)).
\]
Hence, the fiber $U_w$ must be at most countable. This directly contradicts the hypothesis that $U_w$ is an uncountable space for $\mathbb{P}$-a.e. $w \in \Omega$. Therefore, $F_{\mathcal{U}}$ cannot be randomly $cw$-expansive.
\end{proof}

The following definition of sensitivity is a natural adaptation of the concept found in Naif and Kadhim~\cite{naif} to our random dynamical systems framework.

\begin{definition}
A random dynamical system $F_{\mathcal{U}}$ is said to be \emph{random sensitive}  if there exists a positive mensurable random variable $\delta(\omega)$ such that for $\mathbb{P}$-a.e. $\omega \in \Omega$, for any $x \in U_\omega$ and any neighborhood $U \subset U_\omega$ of $x$, there exists $y \in U$ and some $n \in \mathbb{Z}$ satisfying
\[
d\left(f^{n}_{\omega}(x), f_{\omega}^{n}(y)\right) > \delta(\theta^{n} \omega). 
\]
\end{definition}

The relationship between random $cw$-expansiveness and sensitivity is established in the following result:

\begin{proposition}\label{cwexpansivesensitive}
Let $F_{\mathcal U}$ be a random continuum-wise expansive system. Suppose that, for $\mathbb P$-almost every $w\in\Omega$, the fiber $U_w$ is a locally connected compact metric space without isolated points. Then $F_{\mathcal U}$ is random sensitive.
\end{proposition}

\begin{proof}
Let $e = e(w)$ be the expansivity characteristic of the random dynamical system. Fix a positive measurable random variable
$0<\delta(w)<\frac12e(w)$ for $\mathbb P$-almost every $w\in\Omega$.

Consider an arbitrary point $x \in U_w$ and any open neighborhood $U \subset U_w$ of $x$. Since $U_w$ is locally connected, the connected component of $U$ containing $x$, denoted by $V_x$, is an open and connected neighborhood of $x$. Furthermore, since $U_w$ has no isolated points, $V_x$ is a non-trivial connected open set, which implies it contains a non-degenerate subcontinuum $\mathcal{C} \subset V_x \subset U$ such that $x \in \mathcal{C}$.

By random continuum-wise expansiveness, there exists
$n\in\mathbb Z$ such that $\operatorname{diam}(f_w^n(\mathcal C))
>
e(\theta^n(w))$. 

Hence, there exists $a,b\in\mathcal C$ satisfying
$d(f_w^n(a),f_w^n(b)) > e(\theta^n(w))$. By the triangle inequality,
\[
d(f_w^n(a),f_w^n(x))
+
d(f_w^n(x),f_w^n(b))
>
e(\theta^n(w)),
\]
and therefore at least one of the points
$y\in\{a,b\}$
satisfies
\[
d(f_w^n(x),f_w^n(y))
>
\frac12e(\theta^n(w))
>
\delta(\theta^n(w)).
\]
\end{proof}

The following example, originally introduced in \cite{BilbaoOliveiraSantana2024}, illustrates a concrete application where random $cw$-expansiveness yields random sensitivity.

\begin{example}
Let $X$ be a locally connected compact metric space with no isolated points. Let $f_0, f_1 : X \to X$ be homeomorphisms such that $f_0$ is an isometry and $f_1$ is a deterministic continuum-wise expansive homeomorphism with expansivity constant $\delta_0 > 0$ that satisfies the shadowing property (The existence of such examples is well known; see, for instance, \cite{AAV}). 

Consider the full shift space $\Omega = \{0,1\}^{\mathbb{Z}}$ endowed with the standard product topology, and let $\sigma: \Omega \to \Omega$ denote the shift map. We define the skew-product map $F: \Omega \times X \to \Omega \times X$ by
\[
F(\omega, x) = (\sigma(\omega), f_{\omega_0}(x)).
\]
The $n$-th iterate of this map for $n > 0$ is explicitly given by:
\[
F^{n}(\omega, x) = \left(\sigma^n(\omega), f_{\omega_{n-1}} \circ \dots \circ f_{\omega_0}(x)\right).
\]
Thus, $F$ defines a random dynamical system over the base $(\Omega, \mathcal{F}, \mathbb{P})$. Let $\mathbb{P}$ be an ergodic shift-invariant probability measure on $\Omega$ such that 
\[
\mathbb{P}\left(\{\omega \in \Omega : \text{there is only a finite number of symbols } \omega_i = 1\}\right) = 0.
\] 

It was shown in \cite{BilbaoOliveiraSantana2024} that this system $F$ is randomly $cw$-expansive. Since $X$ is locally connected and has no isolated points, Proposition~\ref{cwexpansivesensitive} implies that $F$ is random sensitive.
\end{example}

\section{Shadowing for Random Dynamical Systems}

In this section, we introduce the notion of periodic shadowing for random dynamical systems, which generalizes the classical shadowing property to trajectories that return to their initial fibers. This extension is essential for characterizing the stability of periodic orbits in the random context.

Following \cite{GundlachKifer}, we recall the notions of pseudo-orbit and shadowing for random dynamical systems.

\begin{definition}
Let $F_{\mathcal U}$ be a random dynamical system over
$(\Omega,\theta,\mathbb P)$ and let
$\varepsilon=\varepsilon(w)$ and
$\delta=\delta(w)$ be positive random variables.

\begin{itemize}

\item A sequence
$\{x_k\}_{k\in\mathbb Z}$ with
$x_k\in U_{\theta^k(w)}$
is called a $(w,\delta)$-pseudo-orbit if
\[
d\bigl(f_{\theta^k(w)}(x_k),x_{k+1}\bigr)
\le
\delta(\theta^{k+1}(w))
\]
for every $k\in\mathbb Z$.

\item A finite sequence
$\{x_k\}_{k=0}^{n}$ with
$x_k\in U_{\theta^k(w)}$
is called a finite $(w,\delta)$-pseudo-orbit if
\[
d\bigl(f_{\theta^k(w)}(x_k),x_{k+1}\bigr)
\le
\delta(\theta^{k+1}(w))
\]
for every $0\le k\le n-1$.

\item A point $z\in U_w$ is said to $(w,\varepsilon)$-shadow a $(w,\delta)$-pseudo-orbit $\{x_k\}_{k\in\mathbb Z}$ if
\[
d(f_w^k(z),x_k)
\le
\varepsilon(\theta^k(w))
\]
for every $k\in\mathbb Z$.

\item A point $z\in U_w$ is said to finite $(w,\varepsilon)$-shadow a finite $(w,\delta)$-pseudo-orbit $\{x_k\}_{k=0}^{n}$ if
\[
d(f_w^k(z),x_k)
\le
\varepsilon(\theta^k(w))
\]
for every $0\le k\le n$.

\item $F_{\mathcal U}$ has the shadowing property if, for every positive random variable $\varepsilon$, there exists a positive random variable $\delta$ such that every $(w,\delta)$-pseudo-orbit is $(w,\varepsilon)$-shadowed by the orbit of some point $z\in U_w$.

\item $F_{\mathcal U}$ has the finite shadowing property if, for every positive random variable $\varepsilon$, there exists a positive random variable $\delta$ such that every finite $(w,\delta)$-pseudo-orbit is finite $(w,\varepsilon)$-shadowed by the orbit of some point $z\in U_w$.

\end{itemize}
\end{definition}

The next lemma about uniqueness of shadows is a generalization of the result in \cite{Lee2016}.

\begin{lemma}\label{lemma3}
Let $F_{\U}$ be a random dynamical system $cw$-expansive with characteristic of expansivity $e(w)>0$ and that has random shadowing property. Given a random variable $0<\epsilon(w)< \frac{e(w)}{2}$ and $\delta(w)>0$ corresponds to $\epsilon(w)$ as in the shadowing property definition, then there exists a unique point $z\in U_w$ which $(w,\epsilon)$-shadows a given $(w,\delta)$-pseudo orbit. 
\end{lemma}

\begin{proof}
Let $\{x_n\}_{n \in \mathbb{Z}}$ be a given $(w, \delta)$-pseudo-orbit of $F_{\mathcal{U}}$. Suppose there exists two distinct points $z, y \in U_w$ that $(w, \epsilon)$-shadow the sequence $\{x_n\}_{n \in \mathbb{Z}}$. By the triangle inequality and the definition of shadowing, we have:
\begin{align}
d\left(f_{w}^{i}(z), f_{w}^{i}(y)\right) &\leq d\left(f_{w}^{i}(z), x_i\right) + d\left(x_i, f_{w}^{i}(y)\right) \nonumber \\
& \leq \epsilon(w^i(w)) + \epsilon(\theta^i(w)) \nonumber \\
& = 2\epsilon(\theta^i(w)) < e(\theta^i(w)),
\end{align}
for every $i \in \mathbb{Z}$. 

This uniformly bounded distance implies that if we consider any non-trivial continuum $A \subset U_w$ connecting $z$ and $y$, its image under any iteration would satisfy $\diam(f_w^i(A)) \le 2\epsilon(\theta^i(w)) < e(\theta^i(w))$ for all $i \in \mathbb{Z}$. This implies that $A \subset \Gamma_{e}(w, z)$, which directly contradicts Lemma~\ref{lemma1} since $\Gamma_{e}(w, z)$ cannot contain non-trivial continua. 
\end{proof}

The following result establishes the topological invariance of the random shadowing property, showing that it is preserved under topological conjugacy.

\begin{proposition}\label{proposition3}
Let $F_{\mathcal{U}}$ and $G_{\mathcal{V}}$ be two topologically conjugated random dynamical systems over $(\Omega, \mathcal{F}, \mathbb{P})$. If $F_{\mathcal{U}}$ has the random shadowing property, then $G_{\mathcal{V}}$ also satisfies the random shadowing property.
\end{proposition}

\begin{proof}
Let $\epsilon:\Omega\to(0,\infty)$ be a positive random variable, and let
$\{\pi_w:U_w\to V_w\}_{w\in\Omega}$ be the bundle homeomorphism defining the conjugacy between $F_{\mathcal U}$ and $G_{\mathcal V}$.

Since each fiber is compact, every map $\pi_w$ is uniformly continuous. Hence, for every $w\in\Omega$ there exists $\epsilon_0(w)>0$ such that if $d(x,y)<\epsilon_0(w)$ then $d(\pi_w(x),\pi_w(y))<\epsilon(w)$, for all $x,y\in U_w$. Moreover, since $w\mapsto\pi_w$ is a measurable bundle homeomorphism and the fibers are compact, we may choose $\epsilon_0:\Omega\to(0,\infty)$ measurable.

Since $F_{\mathcal U}$ has the random shadowing property, there exists a positive measurable random variable $\delta_0:\Omega\to(0,\infty)$ such that every $(w,\delta_0)$-pseudo-orbit of $F_{\mathcal U}$ is $(w,\epsilon_0)$-shadowed.

Similarly, since each $\pi_w^{-1}$ is uniformly continuous and
$w\mapsto\pi_w^{-1}$ is measurable, there exists a measurable positive random variable $\delta:\Omega\to(0,\infty)$ such that $d(\pi_w^{-1}(x'),\pi_w^{-1}(y'))<\delta_0(w)$  whenever $d(x',y')<\delta(w)$ for all $x',y'\in V_w$.

Now let $\{y_n\}_{n\in\mathbb Z}$ be a $(w,\delta)$-pseudo-orbit of
$G_{\mathcal V}$. Then for every $n\in\mathbb Z$
\[
d\bigl(g_{\theta^n(w)}(y_n),y_{n+1}\bigr)
\le
\delta(\theta^{n+1}(w)),
\]
for $n\in\mathbb Z$. 

By the choice of $\delta$ and the conjugacy relation
\[
\pi^{-1}_{\theta^{n+1}(w)}
\circ
g_{\theta^n(w)}
=
f_{\theta^n(w)}
\circ
\pi^{-1}_{\theta^n(w)},
\]
we obtain
\begin{equation*}
\begin{aligned}
&
d\!\left(
f_{\theta^n(w)}
(\pi^{-1}_{\theta^n(w)}(y_n)),
\pi^{-1}_{\theta^{n+1}(w)}(y_{n+1})
\right)
\\
&\qquad\le
\delta_0(\theta^{n+1}(w)).
\end{aligned}
\end{equation*}
Hence, $\{\pi^{-1}_{\theta^n(w)}(y_n)\}_{n\in\mathbb Z}$ is a $(w,\delta_0)$-pseudo-orbit of $F_{\mathcal U}$.

By the shadowing property, there exists $x\in U_w$ such that
\[
d\!\left(
f_w^n(x),
\pi^{-1}_{\theta^n(w)}(y_n)
\right)
\le
\epsilon_0(\theta^n(w)),
\qquad n\in\mathbb Z.
\]
Therefore
\[
\begin{aligned}
d\bigl(g_w^n(\pi_w(x)),y_n\bigr)
&=
d\!\left(
\pi_{\theta^n(w)}(f_w^n(x)),
\pi_{\theta^n(w)}
(\pi^{-1}_{\theta^n(w)}(y_n))
\right)  \\
&\le
\epsilon(\theta^n(w)),
\end{aligned}
\]
for every $n\in\mathbb Z$. 

Thus, the point
$z=\pi_w(x)\in V_w$ $(w,\epsilon)$-shadows the pseudo-orbit
$\{y_n\}_{n\in\mathbb Z}$.

Hence, $G_{\mathcal V}$ has the random shadowing property.
\end{proof}

The following result generalizes the classical result by Pilyugin \cite{pilyugin} for the random case, establishing that the finite shadowing property is sufficient to guarantee the shadowing property.

\begin{proposition}\label{finiteSP}
Let $F_{\mathcal U}$ be a random dynamical system whose fibers are compact metric spaces. If $\mathcal{F}_{\U}$ has finite shadowing property, then it also possesses the shadowing property.
\end{proposition}

\begin{proof}
Let $\epsilon = \epsilon(w) > 0$ be a positive random variable. By the finite shadowing property, there exists a there exists a positive measurable random variable $\delta = \delta(w)$ such that any finite $(w, \delta)$-pseudo-orbit is finite $(w, \epsilon)$-shadowed.

Consider a $(w, \delta)$-pseudo-orbit $\{x_k\}_{k \in \Z}$. For each $N \in \mathbb{N}$, consider $\{x_k\}_{k=-N}^{N}$ a $(w, \delta)$-pseudo-orbit of length $2N+1$. By the finite shadowing property, there exists a point $z_N \in U_{\theta^{-N}(w)}$ that finitely $(w,\epsilon)$-shadows this segment, meaning that
\[ d(f_{\theta^{-N}(w)}^j(z_N), x_{-N+j}) \leq \epsilon(\theta^{-N+j}(w)), \quad \text{for } 0 \leq j \leq 2N. \]

For the point $y_N = f_{\theta^{-N}(w)}^N(z_N)$ in $U_w$ we have
\[ d(f_w^k(y_N), x_k) \leq \epsilon(\theta^k(w)), \quad \text{for} -N \leq k \leq N. \]

By compactness of $U_w$ the sequence $\{y_N\}_{N \in \mathbb{N}} \subset U_w$ has a convergent subsequence $\{y_{N_j}\}_{j \in \mathbb{N}}$, such that $\lim_{j \to \infty} y_{N_j} = z \in U_w$.

Fix $k \in \Z$. For any $N_j > |k|$, it holds that 
\[ d(f_w^k(y_{N_j}), x_k) \leq \epsilon(\theta^k(w)). \]

Since $f_w^k: U_w \to U_{\theta^k(w)}$ is continuous, taking the limit 
\[d(f_w^k(z), x_k) \leq \epsilon(\theta^k(w)). \]

Thus the infinite pseudo-orbit $\{x_k\}_{k \in \Z}$ is $(w, \epsilon)$-shadowed by the orbit of $z \in U_w$.
\end{proof}

In the random context, periodicity imposes return on the fibers. A point $x \in U_w$ is called a return point if $f_w^n(x) = x$ for some $n \in \mathbb{N}$. If, in addition, the base map satisfies $\theta^n (w) = w$, then the pair $(w, x)$ is a periodic point for the random dynamical system $F_{\mathcal{U}}$.

The periodic shadowing property was originally introduced by Kościelniak \cite{Koscielniak1996} in the deterministic setting. We generalize this notion to  random dynamical systems. 

\begin{definition}
We say that $F_{\mathcal{U}}$ has the periodic shadowing property if for every random variable $\epsilon(w) > 0$, there exists a random variable $\delta(w) > 0$ such that, whenever $w \in \Omega$ satisfies $\theta^p(w) = w$  for some minimal period $p \in \mathbb{N}$, then every periodic $(w, \delta)$-pseudo-orbit $\{x_k\}_{k \in \mathbb{Z}}$ ($x_{k+p} = x_k$) is $(w, \epsilon)$-shadowed by a periodic orbit, that is, there exists a point $z \in U_w$ satisfying
\begin{enumerate}
    \item $d(f_w^k(z), x_k) < \epsilon(\theta^k(w))$ for all $k \in \mathbb{Z}$;
    \item $f_w^n(z) = z$.
\end{enumerate}
\end{definition}

The requirement that $\theta^p(w) = w$ is intrinsic to the notion of periodicity in random dynamical systems. Indeed, a periodic orbit can only exist over a periodic point of the driving system, since after $p$ iterates the orbit must return to the same fiber. In particular, when the base transformation is the bilateral shift, periodic points are dense, so periodic fibers naturally arise in many standard models of random dynamical systems.

\begin{remark}
For the case of periodic shadowing, the results of the Lemma\ref{lemma3} and Proposition\ref{proposition3} also apply, that is, the uniqueness of shadowing points and invariance under conjugation.
\end{remark}

\begin{theorem}\label{TeoPeSP}
Let $F_{\mathcal U}$ be a random $cw$-expansive dynamical system with the shadowing property. Then $F_{\mathcal U}$ has the periodic shadowing property.
\end{theorem}

\begin{proof}
Let $e = e(w)$ be the characteristic of random $cw$-expansivity of $\mathcal{F}_{\mathcal{U}}$. By the shadowing property, fix a positive meansurable random variable $0 < \epsilon(w) < e(w)/2$, there exists a random variable $\xi(w) > 0$ such that any $(w, \xi)$-pseudo-orbit $\{x_i\}_{i \in \mathbb{Z}}$ is $(w, \epsilon)$-shadowed by some point $z \in U_w$.

Consider a $w\in \Omega$ such that $\theta^p (w) = w$ and a periodic $(w, \xi)$-pseudo-orbit $\{x_i\}_{i \in \mathbb{Z}}$. By definition, the point $z$ shadows $\{x_i\}$, meaning $d(f_w^i(z), x_i) < \epsilon(\theta^i( w))$ for all $i \in \mathbb{Z}$. 

Now, consider the point $y = f_w^p(z)$. We shall show that $y$ is also a $(w, \epsilon)$-shadow for the same pseudo-orbit $\{x_i\}_{i \in \mathbb{Z}}$. Indeed, using the periodicity of the pseudo-orbit and of the base, we have:
\[ 
d(f_w^i(y), x_i) = d(f_w^{i+p}(z), x_{i+p}) < \epsilon(\theta^{i+p} (w)) = \epsilon(\theta^i (w)), \quad \forall i \in \mathbb{Z}.
\]
This shows that $y$ is a $(w, \epsilon)$-shadow of $\{x_i\}_{i \in \mathbb{Z}}$. 
Both $z$ and $y=f_w^p(z)$ are $(w,\varepsilon)$-shadowing points of the same pseudo-orbit. By Lemma~\ref{lemma3}, the shadowing point is unique. Hence, $f_w^p(z)=z$, which implies that $z$ is a periodic point for the random dynamical system, completing the proof.
\end{proof}

A partial converse can be obtained under additional assumptions inspired by the notions introduced in \cite{abbs} and by fiberwise shadowing properties considered in \cite{Azjargal}. To this end, we introduce a random version of topological mixing together with a shadowing property restricted to the orbit of a fixed base point.

\begin{definition}
A random dynamical system $F_{\mathcal{U}}$ is said to be \textit{topologically mixing} if, for any two non-empty random open sets $A, B \subset \mathcal{U}$, for $\mathbb P$-a.e. $w$ there exists an integer $n_0(w)$ such that \[ f_{w}^n (A_{w}) \cap B_{\theta^n(w)} \neq \emptyset \] for all $n\ge n_0(w)$ and $\mathbb{P}$-a.e.  $w \in \Omega$.
\end{definition}

\begin{definition}
For $w\in\Omega$, let
\[
\mathcal O(w)=\{\theta^k(w):k\in\mathbb Z\}
\]
denote the orbit of $w$.

We say that $F_{\mathcal U}$ has the
$w$-shadowing property if, for every positive function
\[
\varepsilon:\mathcal O(w)\to(0,\infty),
\]
there exists a positive function
\[
\delta:\mathcal O(w)\to(0,\infty)
\]
such that every $(w,\delta)$-pseudo-orbit
$\{x_k\}_{k\in\mathbb Z}$ is
$(w,\varepsilon)$-shadowed by some point
$z\in U_w$, namely $d(f_w^k(z),x_k)
<
\varepsilon(\theta^k(w))$ for every $k\in\mathbb Z$.
\end{definition}

\begin{remark}\label{remarkfiniteSP}
The proof of Proposition~\ref{finiteSP} extends verbatim to the $w$-shadowing property. Since the definition is formulated along the orbit of a fixed point $w\in\Omega$, finite shadowing along the fiber orbit is sufficient to guarantee the existence of a full shadowing orbit.
\end{remark}

\begin{theorem}\label{wshadow}
Let $F_{\mathcal U}$ be a topologically mixing bundle random dynamical system.
If $F_{\mathcal U}$ has the periodic shadowing property, then it has the
$w$-shadowing property for every $\theta$-periodic point
$w\in\Omega$.
\end{theorem}

\begin{proof}
Let $w\in\Omega$ be a $\theta$-periodic point with period $p\in\mathbb N$, and let $\epsilon$ be a positive random variable along the orbit of $w$. Let $\delta$ be given by the periodic shadowing property.

Consider a finite $(w,\delta)$-pseudo-orbit $
\{x_0,\dots,x_{k-1}\}$, such that $x_i\in U_{\theta^i(w)}$. Define the open sets
\[
A=
B\!\left(
f_{\theta^{k-1}(w)}(x_{k-1}),
\delta(\theta^k(w))
\right)
\subset U_{\theta^k(w)}
\]
and
\[
C=
B(x_0,\delta(w))
\subset U_w.
\]

Since $F_{\mathcal U}$ is topologically mixing, there exists an integer $N_0\in\mathbb N$ such that
\[
f_{\theta^k(w)}^n(A_{\theta^k(w)})
\cap
B_{\theta^{k+n}(w)}
\neq\emptyset
\]
for all $n\ge N_0$. Choose $N\ge N_0$ such that $k+N$
is a multiple of $p$. Hence, $\theta^{k+N}(w)=w$. By mixing, there exists
$y_0\in A_{\theta^k(w)}$ such that $y_N
=
f_{\theta^k(w)}^N(y_0)
\in B_w$. For each $0\le i\le N$, define $y_i
=
f_{\theta^k(w)}^i(y_0)$. 

Now construct the sequence
\[
\xi=\{s_i\}_{i\in\mathbb Z}
\]
by periodic repetition of  $x_0,\dots,x_{k-1},y_0,\dots,y_{N-1}$. Since $y_0,\dots,y_N$ is a true orbit segment, 
\[d(
f_{\theta^{k-1}(w)}(x_{k-1}),
y_0
)
<
\delta(\theta^k(w))\] and 
\[
d(
f_{\theta^{k+N-1}(w)}(y_{N-1}),
x_0
)
=
d(y_N,x_0)
<
\delta(w)
\]
$\xi$ is a periodic $(w,\delta)$-pseudo-orbit of period $M=k+N$.

By the periodic shadowing property, there exists a periodic point $z\in U_w$ such that $
f_w^M(z)=z$ and
\[
d(f_w^i(z),s_i)
<
\epsilon(\theta^i(w))
\]
for all $i\in\mathbb Z$. In particular,
\[
d(f_w^i(z),x_i)
<
\epsilon(\theta^i(w))
\]
for all $0\le i\le k-1$. Hence, every finite $(w,\delta)$-pseudo-orbit is
$(w,\epsilon)$-shadowed.
By Remark~\ref{remarkfiniteSP},
$F_{\mathcal U}$ has the $w$-shadowing property.
\end{proof}

The assumptions of the converse theorem are natural in the random setting. Periodic shadowing is defined on periodic fibers, since periodic orbits can occur only over periodic points of the driving system. Furthermore, for the bilateral shift, periodic points are dense, making the result applicable to a broad class of random dynamical systems

\subsection{Topological Stability for Random Dynamical Systems}

Consider the compact metric space $X$ with metric $d$. Consider $\mathcal{H}(X)$ the set of all homeomorphisms on $X$. The metric $d_1$ on this set defined as follows: 

if $f$ and $g$ are two homeomorphisms on $X$, we define
\[
d_1 (f,g)  = \max \{ \sup_{x\in X}d(f(x),g(x)) \, , \, \sup_{x\in X}d(f^{-1}(x),g^{-1}(x))\}.
\]

 We denote the set of random dynamical system on $\U$ and over $(\Omega, \theta, \p)$ by $S(\Omega,\U)$.

\begin{definition}
We say that a random dynamical system $F_{\U}$ is random topologically stable if for every positive random variable $\epsilon=\epsilon(w)$ there exists a positive
random variable $\delta=\delta(w)$ such that, for any random dynamical system
$G_{\U}$ satisfying
\[
d_1(f_w,g_w)<\delta(w) \quad \text{for $\mathbb{P}$-a.e. } w\in\Omega,
\]
there exists a measurable family of continuous maps $\{h_w:U_w\to U_w\}_{w\in\Omega}$
such that for $\mathbb{P}$-a.e. $w\in\Omega$, all $x\in U_w$ and all $n\in\mathbb{Z}$,
\begin{itemize}
    \item[i)] $d(h_w(x),x)<\epsilon(w)$
    \item[ii)] $h_{\theta(w)}\circ g_w = f_w \circ h_w $.
\end{itemize}
\end{definition}

Walters proved in \cite{Walters77} that any expansive homeomorphism on a compact metric space with shadowing is topologically stable. In the following we generalize this result.

\begin{theorem}\label{TheA}
Let $F_\U$ be a random $cw$-expansive. If  $F_\U$ has shadowing property, then it is random topologically stable on $S(\Omega,\U)$. 
\end{theorem}
\begin{proof}
Consider the positive random variable $e=e(w)$ the characteristic of random $cw$-expansivity of $F_{\U}$. Fix a positive random variable $\epsilon$  such that $\epsilon(w) < \frac{e(w)}{3}$ for all $w\in \Omega$. By random shadowing property of $F_{\U}$, for $\epsilon(w)$ there exists a positive random variable $\delta = \delta(w)$. 

Consider a random dynamical system $G_{\U}$ such that  $d_1(f_{w},g_{w})<\delta(w)$ for $\mathbb{P}$-a.e $w\in \Omega$.

Given $x\in U_w$ and the sequence $\{g^n_w(x)\}_{n\in \mathbb{Z}}$, note that 
\begin{eqnarray}\label{3}
d(f_{\theta^{n}(w)}(g^n_w(x)), g^{n+1}_{w}(x)) & = & d(f_{\theta^{n}(w)}(g^n_w(x)), g_{\theta^{n}(w)}(g^n_w(x))) \nonumber \\
     & \leq &  \delta(\theta^{n+1}(w)), \nonumber
\end{eqnarray}
so the sequence $\{g^n_w(x)\}_{n\in \Z}$ is a $(w,\delta)$-pseudo orbit of $F_{\U}$. Hence by random shadowing property and Lemma \ref{lemma3}, there exists an unique point that $(w,\epsilon)$-shadows the pseudo orbit $\{g^n_w(x)\}_{n \in \Z}$ which we denote by $h_{w}(x)$.

Thus for $\mathbb{P}$-a.e $w\in \Omega$ we define the map $h_{w}:U_w \to U_{w}$ that satisfies 
\begin{equation}\label{I}
  d(f_{w}^{i}(h_w(x)),g^i_w(x)) < \epsilon(\theta^i(w))  
\end{equation}
for all $i\in \mathbb{Z}$ and $x\in U_w$.

Given $x\in U_w$, from \eqref{3}, we obtain
\[
d(f_{w}^{i+1}(h_w(x)),g^{i+1}_w(x)) < \epsilon(\theta^{i+1}(w)).  
\]
By other side 
\begin{eqnarray}
d(f_{\theta(w)}^{i}(f_w(h_w(x)),g^{i}_{\theta(w)}(g_w(x))) & = & d(f_{w}^{i+1}(h_w(x)),g^{i+1}_w(x)) \nonumber \\
     & \leq & \epsilon (\theta^{i}(\theta(w)))  \nonumber
\end{eqnarray}
for all $i\in \mathbb{Z}$. Thus $f_w(h_w(x))$ satisfies the shadowing condition for the shifted $(\theta(w),\delta)$-pseudo-orbit $\{g_w^{i+1}(x)\}_{i \in \mathbb{Z}}$, the uniqueness of the shadow guaranteed by Lemma \ref{lemma3} implies $h_{\theta(w)}(g_w(x)) = f_w(h_w(x))$.

We now prove that the map $h_w:U_w\to U_w$ is continuous for $\mathbb{P}$-a.e. $w\in\Omega$.

Fix such a $w$ and let $\eta>0$. By Lemma~\ref{lemma2}, there exists $\rho=\rho(w,\eta)>0$ such that if $A\subset U_w$ is a continuum that satisfies
\[
\diam\big(f_w^i(A)\big)\leq e(\theta^i(w)) \quad \text{for all } |i|\leq n(w),
\]
then $\diam(A)<\eta$.

Since $g_w$ is continuous on the compact metric space $U_w$, there exists $\delta_1=\delta_1(w)>0$
such that if $d(x,y)<\delta_1$, then 
\[
d\big(g_w^i(x),g_w^i(y)\big)<\frac{e(\theta^i(w))}{3}
\quad \text{for all } |i|\leq n(w).
\]
By shadowing and the definition of $h_w$, we obtain for all $|i|\leq n(w)$,
\begin{align*}
d\big(f_w^i(h_w(x)),f_w^i(h_w(y))\big)
&\leq d\big(f_w^i(h_w(x)),g_w^i(x)\big) \\
&\quad + d\big(g_w^i(x),g_w^i(y)\big)
+ d\big(g_w^i(y),f_w^i(h_w(y))\big) \\
&< \epsilon(\theta^i(w)) + \frac{e(\theta^i(w))}{3} + \epsilon(\theta^i(w)) \\
&< e(\theta^i(w)).
\end{align*}

Therefore, the set $\{h_w(x),h_w(y)\}$ is contained in a continuum $A\subset U_w$
whose iterates have diameter bounded by $e(\theta^i(w))$ for all $|i|\leq n(w)$.
By Lemma~\ref{lemma2}, it follows that
\[
d\big(h_w(x),h_w(y)\big)<\eta.
\]
\end{proof}

\section{Examples}

We present two examples of random dynamical systems exhibiting shadowing and random cw-expansivity .

\begin{example}

Let $T^2=\mathbb{R}^2/\mathbb{Z}^2$ with the induced metric $d$. Consider two homeomorphisms:

\begin{itemize}
\item $f_0(x,y)=(x+\alpha,y+\beta)\ \mathrm{mod}\ 1$, where $\alpha,\beta\notin\mathbb{Q}$;
\item $f_1(x,y)=(2x+y,x+y)\ \mathrm{mod}\ 1$, induced by the hyperbolic matrix
\[
A=
\begin{pmatrix}
2 & 1\\
1 & 1
\end{pmatrix}.
\]
\end{itemize}

Let $\Sigma\subset\{0,1\}^{\mathbb{Z}}$ be the subshift consisting of the two periodic sequences
\[
\ldots0101\ldots \quad \text{and} \quad \ldots1010\ldots
\]
and consider the corresponding random dynamical system generated by $\{f_0,f_1\}$.

Since the base dynamics is periodic, the system reduces to the deterministic map
$H=f_1\circ f_0$.

Let $\{z_n\}_{n\in\mathbb{Z}}$ be a $(w,\delta)$-pseudo-orbit. Then the subsequence $\{z_{2k}\}$ satisfies
\[
d(H(z_{2k}),z_{2k+2}) \le C\delta,
\]
for some constant $C>0$. Choosing $\delta$ small enough, $\{z_{2k}\}$ is a pseudo-orbit of $H$, hence it is $\varepsilon$-shadowed by a true orbit $\{x_{2k}\}$ of $H$. Setting $x_{2k+1}=f_0(x_{2k})$ yields a full orbit satisfying
\[
d(f_w^n(x),z_n)<\gamma \quad \forall n\in\mathbb{Z}.
\]
Therefore, the RDS has random shadowing property.

Let $C\subset T^2$ be a non-degenerate continuum and $\widetilde{C}$ a lift to $\mathbb{R}^2$. Since the hyperbolic splitting of $A$ satisfies
\[
\mathbb{R}^2 = E^u \oplus E^s,
\]
at least one projection has positive diameter. If $\operatorname{diam}(\pi_u(\widetilde{C}))>0$, then
\[
\operatorname{diam}(H^k(C)) \ge C_0 \lambda_u^k \to \infty.
\]
Hence $C$ eventually exceeds the injectivity radius of the covering, implying random continuum-wise expansivity.

Therefore, by Theorem \ref{TeoPeSP}, the system satisfies the random periodic shadowing property. Moreover, applying Theorem \ref{TheA}, this random dynamical system is random topologically stable.

\end{example}

\begin{example}
Let $f:M\to M$ be an Anosov diffeomorphism on a compact Riemannian manifold $M$. Consider a random perturbation of the form
\[
f_w(x)=\exp_{f(x)}\big(\varepsilon \phi_w(x)\big),
\]
where $\phi_w:M\to T_{f(x)}M$ is 
$C^1$ with 
$||\phi_w||_{C^1}\leq 1$ and 
$\varepsilon >0$ is sufficiently small.

Let $(\Omega,\mathcal F,\mathbb P,\theta)$ be a metric dynamical system. The associated cocycle is defined by
\[
f_w^n
=
f_{\theta^{n-1}(w)}\circ \cdots \circ f_w.
\]

For $\varepsilon>0$ small enough, there exists a measurable splitting
\[
T_xM = E^s(w,x)\oplus E^u(w,x)
\]
which is invariant under $Df_w$. Moreover, there exist constants $C>0$ and $0<\lambda<1$ such that
\[
||Df_w^n v^s||\leq C\lambda^n||v^s||,
\quad v^s\in E^s(w,x),\ n\geq 0,
\]
and an analogous estimate holds for unstable directions in backward time.

The uniform hyperbolicity of the cocycle implies that the associated skew-product
\[
F(w,x)=(\theta(w),f_w(x))
\]
admits a random hyperbolic structure. By standard results in random dynamical systems theory it follows that $F$ satisfies:

\begin{itemize}
\item the random shadowing property;
\item expansivity of the skew-product dynamics;
\item random continuum-wise expansivity.
\end{itemize}

Therefore, $F$ satisfies the random periodic shadowing property by Theorem \ref{TeoPeSP}, and is random topologically stable by Theorem \ref{TheA}.

\end{example}

\bibliographystyle{amsplain}

\end{document}